\documentclass[10pt,reqno]{amsart}

\usepackage[english]{babel}
\usepackage{amsmath,amsthm,amssymb,amscd, enumerate, color, tikz}
\usetikzlibrary{matrix,arrows,positioning,calc}
\usepackage[all]{xy}

\usepackage{hyperref}
\hypersetup{
    colorlinks=true,
    citecolor=blue,
    linktocpage=true,
}

\newtheorem{theorem}{Theorem}[section]
\newtheorem{definition}[theorem]{Definition}
\newtheorem{lemma}[theorem]{Lemma}
\newtheorem{proposition}[theorem]{Proposition}
\newtheorem{corollary}[theorem]{Corollary}
\newtheorem{remark}[theorem]{Remark}

\newtheorem{example}[theorem]{Example}

\newtheorem{question}[theorem]{Question}

\newcommand{\C}{\mathcal{C}}
\newcommand{\T}{\mathcal{T}}
\newcommand{\F}{\mathcal{F}}

\newcommand{\D}{\mathcal{D}}
\newcommand{\U}{\mathcal{U}}
\newcommand{\Ut}{\mathcal{U}_{\mathbf{t}}}
\newcommand{\Gen}{\text{Gen}}

\newcommand{\Hom}{\text{Hom}}

\newcommand{\Ker}{\text{Ker}}

\newcommand{\tCG}{{$t$CG}}
\newcommand{\rightarrowdbl}{\rightarrow\mathrel{\mkern-14mu}\rightarrow}

\begin{document}

\title[$t$CG Torsion Pairs]{{\tCG} torsion pairs}

\author[D. Bravo]{Daniel Bravo}  
\address{
 Instituto de Ciencias F\'isicas y Matem\'aticas \\ Facultad de Ciencias \\
 Universidad Austral de Chile \\
 Valdivia, Chile}
\email{daniel.bravo@uach.cl}
\thanks{The first author was partially supported by CONICYT + FONDECYT/Iniciaci\'on + 11130324.}

\author[C. E. Parra]{Carlos E. Parra }  
\address{
 Instituto de Ciencias F\'isicas y Matem\'aticas \\ Facultad de Ciencias \\
 Universidad Austral de Chile \\
 Valdivia, Chile}
\email{carlos.parra@uach.cl}
\thanks{The second author was partially supported by CONICYT + FONDECYT/Iniciaci\'on + 11160078.}



\date{\today}

\begin{abstract}
 We investigate conditions for when the $t$-structure of Happel-Reiten-Smal\o \ associated to a torsion pair is a compactly generated $t$-structure. The concept of a {\tCG} torsion pair is introduced and for any ring $R$, we prove that $\mathbf{t}=(\T,\F)$ is a {\tCG} torsion pair in $R\text{-Mod}$ if, and only if, there exists, $\{T_{\lambda}\}$ a set of finitely presented $R$-modules in $\mathcal{T}$, such that $\mathcal{F}=\bigcap \text{Ker}({\Hom}_{R}(T_{\lambda},?))$. We also show that every {\tCG} torsion pair is of finite type, and show that the reciprocal is not true. Finally, we give a precise description of the $t$CG torsion pairs over Noetherian rings and von Neumman regular rings.
\end{abstract}


\maketitle

\tableofcontents

\section{Introduction}

The notion of a \emph{torsion pair} was introduced in the sixties by Dickson (see \cite{D}) in the setting of abelian categories, generalizing the classical notions for abelian groups. Since then, torsion pairs have found many applications in the study of localizations, tilting theory, category theory, etc. Indeed, an important class of examples of torsion pairs in triangulated categories is provided by the concept of \emph{$t$-structure}  introduced by Beilinson, Bernstein and Deligne \cite{BBD}, in their study of the perverse sheaves over an analytic or algebraic variety stratified by some closed subsets.   This notion allows us to associate, to an object of an arbitrary triangulated category, its corresponding ``objects of homology'', which belong to some abelian subcategory of such triangulated category. Such subcategory is called the heart of the $t$-structure.

In the  nineties, Happel, Reiten and Smal\o\ observed that there is a natural way to associate a $t$-structure to the derived category of a given abelian category endowed with a torsion pair (see \cite{HRS}). The Happel-Reiten-Smal\o\  $t$-structure is perhaps the most well known $t$-structure for triangulated categories. Nevertheless, other $t$-structures, such as the compactly generated $t$-structures, have also been well documented in the literature (see \cite{AJS}, \cite{AJSo}, \cite{PS4}). Some of these compactly generated $t$-structures,  when certain conditions on the ambient triangulated category are imposed,  have an explicit description for the co-aisle  (see \cite{AJSo}). For these two types of $t$-structures, several authors have investigated conditions for when the heart of such $t$-structures is a  Grothendieck category or a module category (see \cite{CGM}, \cite{CMT}, \cite{HKM}, \cite{MT}, \cite{PS1}, \cite{PS2}, \cite{PS3}, \cite{PS4}). 

In particular, \cite{PS4} shows that over a commutative Noetherian ring $R$, the heart of almost every compactly generated $t$-structure in $\mathcal{D}(R)$, the derived category of the ring $R$, is a Grothendieck category.  On the other hand, Theorem 3.7 in \cite{PS1} shows that a countable direct limit of exact sequences in the heart of a compactly generated $t$-structure is always exact. Recall that over a Grothendieck category, direct limits of exact sequences are exact.

Hence, given any ring $R$, the following question seems natural to ask: is the heart of a compactly generated $t$-structure in $\mathcal{D}(R)$ an AB5  category? We tackle this question for the $t$-structure of Happel-Reiten-Smal\o.  The main goal of this article is to provide a positive answer to this question, through the concept of \emph{{\tCG} torsion pairs} (see Definition \ref{tCG-definition}), and to study the relation of torsion pairs of finite type and the {\tCG} torsion pairs. 

The organization of this paper is as follows. In Section \ref{S:Prelim} we give all the preliminaries and terminology needed in the rest of the paper. Section \ref{S:tCG} introduces the reader to the notion of a {\tCG} torsion pair, and contains the main result of this article, namely Theorem \ref{Teo. tCG}, which gives a characterization of the {\tCG} torsion pairs. This characterization result, then allows us to describe the  {\tCG} torsion pairs  over Noetherian rings (see Theorem \ref{Teo tCG noetherian}). For coherent rings, we establish  an injective function between the set of the torsion pairs in $fp(R\text{-Mod})$ and the set of the {\tCG} torsion pairs (see Theorem \ref{Teo. Coherent}). 
In Section \ref{S:Construction}, we study the relation between the {\tCG} torsion pair and the left constituent pair of a TTF triple in $R\text{-Mod}$. In this direction, we obtain an example of a torsion pair of finite type that is not a $t$CG torsion pair. Finally, over a von Neumann regular ring, we show that the only $t$CG torsion pairs   are left constituent pairs of a TTF-triple.

\section{Preliminares and terminology} \label{S:Prelim}

The concepts that we shall introduce in this section are applied in the case of module categories, but sometimes we will use them in the more general context of \emph{Abelian categories} and it is in this context that we define them. Hence in what follows, $\mathcal{A}$ will denote an AB3 abelian category, that is, $\mathcal{A}$ is an abelian category with coproducts. Let $X$ and $V$ be objects of $\mathcal{A}$. We say that $X$ is $V$-\emph{generated}  when there is an epimorphism  ${V^{(I)} \rightarrowdbl X}$, for some set $I$. We will denote by $\Gen(V)$  the class of $V$-generated  objects.  
For a class of objects $\mathcal{S}$ in $\mathcal{A}$, we will use the following notation $\mathcal{S}^{\perp}:=\left\{X \in \mathcal{A} : \text{Hom}_{\mathcal{A}}(S,X)=0, \text{ for all }S \in \mathcal{S}\right\}$ and $^{\perp}\mathcal{S}:=\left\{X \in \mathcal{A} : \text{Hom}_{\mathcal{A}}(X,S)=0, \text{ for all }S \in \mathcal{S}\right\}$.

A \emph{torsion pair} in $\mathcal{A}$ is a pair $\mathbf{t}=(\mathcal{T,F})$ of full subcategories of $\mathcal{A}$ satisfying the following two conditions:
\begin{enumerate}
\item  $\T={}^{\perp}\F$;
\item  For each object $X$ of $\mathcal{A}$, there is an exact sequence
\[
0 \longrightarrow  T_X \longrightarrow X \longrightarrow F_X \longrightarrow 0
\]
\end{enumerate}
where $T_X \in \T$ and $F_X\in \F$. The objects $T_{X}$ and $F_X$ in the previous exact sequence are uniquely determined, up to isomorphism, and the assignment $X \leadsto T_X$ (resp. $X \leadsto F_X$) underlies a functor ${t: \mathcal{A} \rightarrow \T }$ (resp. ${(1:t):\mathcal{A} \rightarrow  \F})$, which is right (resp. left) adjoint to the inclusion functor ${\T \hookrightarrow \mathcal{A}}$ (resp. ${\F  \hookrightarrow  \mathcal{A}}$). 

The composition ${ \mathcal{A} \xrightarrow{t} \T \hookrightarrow  \mathcal{A}}$ (resp. ${ \mathcal{A} \xrightarrow{(1:t)} \F \hookrightarrow \mathcal{A}}$), which we will still denote by $t$ (resp. $(1:t)$), is called the \emph{torsion radical} (resp. \emph{torsion coradical}) associated to $\mathbf{t}$. The torsion pair $\mathbf{t}=(\mathcal{T,F})$ is called \emph{hereditary} when $\T$ is closed under taking subobjects in $\mathcal{A}$. 

In the case of $R$-Mod, where $R$ is a unital associative ring, we are interested in investigating a particular case of hereditary torsion pairs, namely the {TTF-triples}. A \emph{TTF-triple} is a triple of classes $(\C, \T, \F)$ in $R$-Mod such that $(\C,\T)$ and $(\T,\F)$ are torsion pairs. In this case \cite[Proposition VI.6.12]{S} characterizes the TTF-triples in terms of two-sided idempotent ideal of $R$. More precisely, $(\C,\T,\F)$ is a TTF-triple if and only if there exists $\mathfrak{a}$, a two-sided idempotent ideal of  $R$, such that $\mathcal{C}=\text{Gen}(\mathfrak{a})=\{C\in R\text{-Mod}: \mathfrak{a}C=C\}$, $\mathcal{T}=\{T\in R\text{-Mod}: \mathfrak{a}T=0\}$ and $\mathcal{F}=\text{Ker}(\text{Hom}_{R}({R}/{\mathfrak{a}},?))$. In reference to the ideal $\mathfrak{a}$, we will denote such TTF-triple as $(\C_\mathfrak{a},\T_\mathfrak{a},\F_\mathfrak{a})$.

In the sequel, we let $(\mathcal{D},?[1])$ be  a triangulated category, and denote the triangles in $\D$ by ${X  \rightarrow Y  \rightarrow Z  \xrightarrow{+}  }$. An additive functor ${H:\D  \rightarrow \mathcal{A}}$, where $\mathcal{A}$ is an abelian category, is called \emph{cohomological}, if for any triangle ${X  \rightarrow Y  \rightarrow Z  \xrightarrow{+}  }$ in $\D$, we get the following long exact sequence in $\mathcal{A}$:
\[
{\cdots  \longrightarrow H^{n-1}(Z) \longrightarrow H^{n}(X) \longrightarrow H^{n}(Y) \longrightarrow H^{n}(Z) \longrightarrow \cdots},
\]
where $H^{n}(?):=H \circ (?[n])$, for each integer $n$. A pair $(\mathcal{U,V})$ of full subcategories in $\D$, is called a \emph{$t$-structure}, when both class are closed under  direct summands, and satisfies the following assertions:
\begin{enumerate}
\item $\Hom_{\D}(U,V[-1])=0$, for all $U\in \mathcal{U}$ and for all $V\in \mathcal{V}$.
\item $\mathcal{U}[1] \subseteq \mathcal{U}$.
\item For each $X \in \text{Ob}(\D)$, there is a triangle ${U_X \rightarrow X \rightarrow V_X \xrightarrow{+}}$ in $\D$, where $U_X\in \mathcal{U}$ and $V_X \in \mathcal{V}[-1]$.
\end{enumerate}
In this case, $\mathcal{U}$ is called the \emph{aisle} of the t-structure, and $\mathcal{V}$ is called the \emph{co-aisle} of the t-structure. Moreover, $\mathcal{V}=\mathcal{U}^{\perp}[1]$ and $\mathcal{U}= {}^{\perp}(\mathcal{V}[-1])={}^{\perp}(\mathcal{U}^{\perp})$, hence we will denoted such t-structure by $(\mathcal{U},\mathcal{U}^{\perp}[1])$.

If $(\mathcal{U},\mathcal{U}^{\perp}[1])$ is a t-structure in $\D$, then the full subcategory $\mathcal{H}=\mathcal{U} \cap \mathcal{V}=\U \cap \U^{\perp}[1]$ is an abelian category, called the \emph{heart} of the $t$-structure and there is a cohomological functor $H^{0}_{\mathcal{U}}:\D \rightarrow \mathcal{H}$ (see \cite{BBD}). If $\D=\D(R)$ and $\mathcal{S}$ is a set of objects in $\D$, the smallest full subcategory of $\D$, containing $\mathcal{S}$, closed under coproducts, extensions and positive shifts is the aisle of a $t$-structure (cf. \cite[Proposition 3.2]{AJSo}). In that case, if $(\U,\U^{\perp}[1])$ denotes such $t$-structure, then $\U^{\perp}$ consists of the $Y\in \D$ such that $\Hom_{\D}(S[n],Y)=0$, for all $S\in \mathcal{S}$ and for all integers $n\geq 0$. In this case, we will write that $\U=\text{aisle}(\mathcal{S})$, moreover, the $t$-structure $(\text{aisle}(\mathcal{S}),\text{aisle}(\mathcal{S})^{\perp}[1])$ is called \emph{compactly generated} if $\mathcal{S}$ consist of compact objects (i.e. for each $S\in \mathcal{S}$ the functor $\Hom_{\D}(S,?)$ commutes with coproducts) and we say that $\mathcal{S}$ is a \emph{set of compact generators} of the aisle. Recall that the compact objects of $\D(R)$ are the complexes which are quasi-isomorphic to bounded complexes of finitely generated projective modules (see \cite{R}).

For the rest of this section, we assume that $R$ is a commutative Noetherian ring and we denote by $\text{Spec}(R)$ its spectrum. A subset $Z$ of $\text{Spec}(R)$ is \emph{stable under specialization} if, for any pair of prime ideals $\mathfrak{p} \subseteq \mathfrak{q}$, with $\mathfrak{p}\in Z$, it holds that $\mathfrak{q}\in Z$. Equivalently, $Z$ is a union of closed subsets of $\text{Spec}(R)$, with respect to the Zariski topology of $\text{Spec}(R)$. Such a subset will be called \emph{sp-subset} in the sequel. The typical example is the support of an $R$-module N, denoted by $\text{Supp}(N)$, which consists of the prime ideals $\mathfrak{p}$ such that $N_\mathfrak{p}:=R_{\mathfrak{p}} \otimes_{R} N\neq 0$. We have the following description for hereditary torsion pairs in $R\text{-Mod}$.

\begin{proposition}{\cite[Chapter VI \S5,6]{S}} \label{Prop:sp-subset}
The assignment $Z \leadsto (\T_{Z},\T_{Z}^{\perp})$ defines a bijection between the sp-subsets of $\text{Spec}(R)$ and the hereditary torsion pairs in $R\text{-Mod}$, where $\T_{Z}$ is the class of the $R$-modules $T$ such that $\text{Supp}(T)\subseteq Z$. Its inverse takes $(\T,\T^{\perp})$ to the set $Z_{\T}$ of prime ideals $\mathfrak{p}$ such that $R/\mathfrak{p}$ is in $\T$.
\end{proposition}

A \emph{filtration by supports} of $\text{Spec}(R)$ is a decreasing map ${\phi:\mathbb{Z} \rightarrow \text{P(Spec}(R))}$, such that $\phi(i)$ is an sp-subset for each $i\in \mathbb{Z}$; here $\text{P(Spec}(R))$ is the power set of $\text{Spec}(R)$. We will refer to a filtration by supports of $\text{Spec}(R)$ simply by an sp-filtration of $\text{Spec}(R)$. In \cite{AJS}, Alonso, Jerem\'ias and Saor\'in associated to each sp-filtration ${\phi: \mathbb{Z} \rightarrow \text{P(Spec}R)}$, the $t$-structure $(\U_{\phi},\U_{\phi}^{\perp}[1])$ in $\mathcal{D}(R)$, where $\U_{\phi}:=\text{aisle}(\{R/\mathfrak{p}[-i]: i\in \mathbb{Z} \text{ and } \mathfrak{p} \in \phi(i)\})$.

\section{{\tCG} torsion pairs} \label{S:tCG}
Throughout this section, $R$-Mod is the category of left $R$-modules and $\mathbf{t}=(\mathcal{T,F})$ is a torsion pair in $R\text{-Mod}$. The \emph{$t$-structure of Happel-Reiten-Smal\o} in $\D(R)$ associated to the torsion pair $\mathbf{t}$, is  given by $(\Ut,\Ut^{\perp}[1]):=(\Ut,\mathcal{W}_{\mathbf{t}})$, where:
\[
{\Ut=\{X \in \D^{\leq 0}(R) :  H^{0}(X)\in \T \} \quad  \text{and} \quad \mathcal{W}_{\mathbf{t}}=\{X \in \D^{\geq -1}(R)  :  H^{-1}(X)\in \F \} .}
\]
The heart $\mathcal{H}_{\mathbf{t}}$ of this $t$-structure consists of the complexes $M\in \D^{[-1,0]}(R)$ such that $H^{-1}(M)\in \F$ and $H^{0}(M)\in \T$. 

\begin{definition} \label{tCG-definition}
Let $\mathbf{t}=(\mathcal{T,F})$ be a torsion pair in $R\text{-Mod}$. We say that $\mathbf{t}$ is a \emph{{\tCG} torsion pair} when $(\mathcal{U}_{\mathbf{t}},\mathcal{U}_{\mathbf{t}}^{\perp}[1])$ is a compactly generated $t$-structure.
\end{definition}

\begin{example} \label{ejemplo-sp}
Let $R$ be a commutative Noetherian ring. In $R$-Mod, the {\tCG} torsion pairs coincide with the hereditary torsion pairs. Indeed, if $\mathbf{t}=(\mathcal{T,F})$ is a hereditary torsion pair, then  Proposition \ref{Prop:sp-subset} says that there exists $Z$ an sp-subset of $\text{Spec}(R)$ such that $\mathbf{t}=(\mathcal{T,F})=(\T_Z,\T_Z^{\perp})$. Now, we define an sp-filtration as follows:  
\[
\phi(n)= \left\{ \begin{array}{lcc}
             \emptyset &   if & n > 0 \\
              Z &  if & n=0 \\
              \text{Spec}(R) &  if  & n<0 .
             \end{array}
   \right.
\]
By \cite[Theorem 3.11]{AJS}, we obtain that $\mathcal{U}_{\phi}=\{X\in \D(R) : \text{Supp}(H^{j}(X))\subseteq \phi(j), \text{ for all j }\in \mathbb{Z}\}=\Ut$. Now using \cite[Theorem 3.10]{AJS}, we get that $(\mathcal{U}_{\phi},\mathcal{U}_{\phi}^{\perp}[1])=(\Ut,\Ut^{\perp}[1])$ is a compactly generated $t$-structure.

On the other hand, let $\mathbf{t}=(\mathcal{T,F})$ be a {\tCG} torsion pair. By \cite[Theorem 3.10]{AJS}, there exists an sp-filtration $\phi:\mathbb{Z} \rightarrow { \text{P}(\text{Spec}(R))}$, such that $\mathcal{U}_{\mathbf{t}}=\mathcal{U}_{\phi}$. From the description of $\mathcal{U}_{\phi}$, we see that $\phi(k)=\emptyset$, for all integers $k>0$ and $\phi(k)=\text{Spec}(R)$, for all integers $k\leq -1$. Moreover, $\mathcal{T}=\mathcal{T}_{0}$, where $(\mathcal{T}_{0},\mathcal{F}_{0})$ is the hereditary torsion pair associated to the sp-subset $\phi(0)$.
\end{example}

The following result characterizes  the {\tCG}  torsion pairs in terms of the torsion pair in $R\text{-Mod}$. We will use the following notation: for each $M \in R$-Mod, we denote by $M[0]$ the complex concentrated in degree $0$ with $M$, and zero everywhere else.

\begin{theorem}\label{Teo. tCG}
Let $\mathbf{t}=(\mathcal{T,F})$ be a torsion pair in $R$-Mod. Then, $\mathbf{t}$ is a {\tCG} torsion pair, if and only if, there exists a set $\{T_{\lambda}\}_{\lambda \in \Lambda}$  of finitely presented $R$-modules in $\mathcal{T}$, such that $\mathcal{F}={\bigcap}_{\lambda \in \Lambda} \text{Ker}(\text{\Hom}_{R}(T_{\lambda},?))$. 
\end{theorem}

\begin{proof}
Let $\mathcal{S}=\{S_{\lambda}\}_{\lambda \in \Lambda}$ be a set of compact generators of the aisle $\mathcal{U}_{\mathbf{t}}=\{X \in D^{\leq 0}(R): H^{0}(X)\in \mathcal{T}\}$. Without loss of generality, we can assume that each $S_{\lambda}$ is of the form:
\[
{
S_\lambda:= \cdots \xrightarrow{} 0 \xrightarrow{} P_{\lambda}^{n_{\lambda}} \xrightarrow{d_{\lambda}^{n_{\lambda}}}   \cdots \xrightarrow{} P^{0}_{\lambda} \xrightarrow{d^{0}_{\lambda}}  \cdots \xrightarrow{d_{\lambda}^{m_{\lambda}-1}}  P^{m_{\lambda}}_{\lambda} \xrightarrow{}  0 \xrightarrow{} \cdots
},
\]
where each $P_{\lambda}^{k}$ is a finitely generated projective $R$-module (see \cite{R}). Since each $S_{\lambda}$ is in $D^{\leq 0}(R)$,  we obtain that  for $m_{\lambda} \geq 2$, the following exact sequence is split:
\[
{
0 \xrightarrow{} \text{Ker}(d^{m_{\lambda}-1})=\text{Im}(d_{\lambda}^{m_{\lambda}-2}) \xrightarrow{} P_{\lambda}^{m_{\lambda}-1} \xrightarrow{} \text{Im}(d_{\lambda}^{m_{\lambda}-1})=P^{m_{\lambda}}_{\lambda} \xrightarrow{} 0
}.
\]
Thus $\text{Im}(d^{m_{\lambda}-2}_{\lambda})$ is a finitely generated projective $R$-module. Using this argument in a recursive way, we obtain that $\text{Im}(d^{0}_{\lambda})$ is also a finitely generated projective $R$-module and thus, so is $\text{Ker}(d^{0}_{\lambda})$. Next, since $\text{Im}(d_{\lambda}^{-1})$ is a finitely generated $R$-module, the following exact sequence shows that $H^{0}(S_{\lambda})$ is a finitely presented $R$-module:
\[
{
0 \xrightarrow{} \text{Im}(d^{-1}_{\lambda}) \xrightarrow{} \text{Ker}(d_{\lambda}^{0}) \xrightarrow{} H^{0}(S_{\lambda}) \xrightarrow{} 0
}.
\]

We now check that   $\mathcal{F}={\bigcap}_{\lambda \in \Lambda} \text{Ker}(\text{Hom}_{R}(H^{0}(S_{\lambda}),?))$. Indeed, if we fix an $R$-module $M$ in ${\bigcap}_{\lambda \in \Lambda} \text{Ker}(\text{Hom}_{R}(H^{0}(S_{\lambda}),?))$, then it is clear that $t(M)$ also is in ${\bigcap}_{\lambda \in \Lambda} \text{Ker}(\text{Hom}_{R}(H^{0}(S_{\lambda}),?))$. Therefore: 
\[
\begin{array}{rl}
0&= \text{Hom}_{R}(H^{0}(S_{\lambda}),t(M)) \\  
    & \cong \text{Hom}_{\mathcal{D}(R)}(H^{0}(S_{\lambda})[0],t(M)[0]) \\
    & \cong \text{Hom}_{\mathcal{D}(R)}(S_{\lambda},t(M)[0]),
\end{array}
\]
where the last isomorphism follows by applying the contravariant cohomological functor ${\Hom}_{\mathcal{D}(R)}(?,t(M)[0])$ to the canonical triangle obtained from the $t$-structure $\left( D^{\leq -1}(R),D^{\geq -1}(R) \right)$:
\[
{
\tau^{\leq -1}(S_{\lambda}) \xrightarrow{} S_{\lambda} \xrightarrow{} H^{0}(S_{\lambda})[0] \xrightarrow{+} 
}.
\]
It follows that ${\Hom}_{\mathcal{D}(R)}(S_{\lambda}[n],t(M)[0])=0$ for all $S_{\lambda}\in \mathcal{S}$ and integers $n \geq 0$. This implies that $t(M)[0]\in \mathcal{U}_{\mathbf{t}} \cap \mathcal{U}_{\mathbf{t}}^{\perp}=\{ 0\}$ and therefore $M\in \mathcal{F}$.

Conversely, for each $\lambda$, we will denoted by, $S_{\lambda}$, the complex:
\[
\begin{tikzpicture}[node distance=1.35cm]
\node (S) {$S_{\lambda}:=$};
\node (dl) [right of=S] {$\cdots$};
\node (0l) [right of=dl] {$0$};
\node (Rn) [right of=0l] {$R^{(n_{\lambda})}$};
\node (K) [above right of=Rn] {$K_{\lambda}$};
\node (Rm) [below right of=K] {$R^{(m_{\lambda})}$};
\node (0r) [right of=Rm] {$0$};
\node (dr) [right of=0r] {$\cdots$,};
\draw[->] (dl) -- (0l);
\draw[->] (0l) -- (Rn);
\draw[->] (Rn) -- node[above]{\scriptsize $d_{\lambda}$} (Rm);
\draw[->] (Rm) -- (0r);
\draw[->] (0r) -- (dr);
\draw[->>] (Rn) -- (K);
\draw[right hook->] (K) -- (Rm);
\end{tikzpicture}
\]
where $m_\lambda, n_{\lambda}$ are positive integers, $R^{(m_{\lambda})}$ is in degree 0, and $K_{\lambda}$ is the finitely generated $R$-module given by the kernel of the epimorphism ${R^{(m_\lambda)} \rightarrowdbl T_{\lambda}}$. Note that $\mathcal{S}_0 :=\{S_{\lambda}\} \cup \{ R[1]\}$ is a set of compact complexes in $\mathcal{U}_{\mathbf{t}}$, and therefore $\text{aisle}(\mathcal{S}_0)\subseteq \mathcal{U}_{\mathbf{t}}$. 
On the other hand, let $X$ be a complex in $\text{aisle}(\mathcal{S}_0)^{\perp}$. Then we obtain that
\[
0=\text{Hom}_{\mathcal{D}(R)}((R[1])[n],X)=\text{Hom}_{\mathcal{D}(R)}(R,X[-1-n])=H^{-1-n}(X),
\]
for all integers $n\geq 0$, showing that $X\in \mathcal{D}^{\geq 0}(R)$. 

Now, applying the cohomological functor ${\Hom}_{\mathcal{D}(R)}(S_{\lambda},?)$ to the triangle 
\[
{H^{0}(X)[0] \rightarrow X \rightarrow \tau^{>0}(X) \xrightarrow{+}},
\]
gives that $\text{Hom}_{\mathcal{D}(R)}(S_{\lambda},X)\cong \text{Hom}_{\mathcal{D}(R)}(S_{\lambda},H^{0}(X)[0])$. But,  
\[
\begin{array}{rl}
0= & \text{Hom}_{\mathcal{D}(R)}(S_{\lambda},X) \\ 
 \cong & \text{Hom}_{\mathcal{D}(R)}(S_{\lambda},H^{0}(X)[0]) \\ 
 \cong & \text{Hom}_{\mathcal{D}(R)}(H^{0}(S_{\lambda})[0],H^{0}(X)[0]) \\ 
 = & \text{Hom}_{R}(T_{\lambda},H^{0}(X)).
\end{array}
\] 

Since $\F={\bigcap}_{\lambda \in \Lambda} \text{Ker}(\text{Hom}_{R}(T_{\lambda},?))$, it follows that $H^{0}(X)\in \mathcal{F}$. Therefore, $\text{aisle}(\mathcal{S}_0)^{\perp} \subseteq \mathcal{U}_{\mathbf{t}}^{\perp}$,  and since $\mathcal{U}_{\mathbf{t}}^{\perp} \subseteq \text{aisle}(\mathcal{S}_0)^{\perp}$, then we obtain that $(\mathcal{U}_{\mathbf{t}},\mathcal{U}_{\mathbf{t}}^{\perp}[1])=(\text{aisle}(\mathcal{S}_0),\text{aisle}(\mathcal{S}_0)^{\perp}[1])$ is a compactly generated $t$-structure. 
\end{proof}

\begin{remark} \label{isom-classes-fin-pres}
Using the fact that the isomorphism classes of finitely presented modules form a set, along with Theorem \ref{Teo. tCG}, we get that the collection of the {\tCG} torsion pairs actually form a set, that we  denote by $\mathbf{t}\mathcal{CG}(R)$.
\end{remark}

The simplest version of Theorem \ref{Teo. tCG} is when the set $\{T_{\lambda}\}_{\lambda \in \Lambda}$ is a singleton. In such case $\F=\text{Ker}(\text{Hom}_{R}(T_{\lambda},?))$. One way to obtain this condition over $\F$ is when $\T=\text{Gen}(T_{\lambda})$, for some finitely presented $T_{\lambda}$; for example  if $\mathcal{H}_{\mathbf{t}}$ is a module category (see \cite[Lemma 3.2]{PS3}).

\begin{corollary}
Let $R$ be a ring. Then, every torsion pair $\mathbf{t}=(\mathcal{T,F})$ in $R$-Mod such that $\mathcal{T}=\text{Gen}(V)$, for some finitely presented $R$-module $V$, is a {\tCG} torsion pair.
\end{corollary}

However, the reciprocal of the previous result is not true; see Remark \ref{no-tCG-fin-pres}.

\begin{corollary}\label{cor. tCG implies F}
Let $R$ be a ring and $\mathbf{t}=(\mathcal{T,F})$ be a torsion pair in $R$-Mod. If $\mathbf{t}$ is a {\tCG} torsion pair, then $\mathcal{H}_{\mathbf{t}}$ is a Grothendieck category.
\end{corollary}
\begin{proof}
Note that $\mathcal{F}$ is closed under direct limits if $\mathbf{t}$ is a {\tCG} torsion pair. The result now follows from \cite[Theorem 1.2]{PS2}.
\end{proof}

\begin{definition}
A torsion pair $\mathbf{t}=(\mathcal{T,F})$ in $R$-Mod is said to be of \emph{finite type}, if $\mathcal{F}=\varinjlim{\mathcal{F}}$.
\end{definition}

\begin{corollary}
Let $R$ be a ring and let $\mathbf{t}=(\mathcal{T,F})$ be a hereditary torsion pair in $R$-Mod. Then, $\mathbf{t}$ is a {\tCG} torsion pair if, and only if, $\mathbf{t}$ is of finite type.
\end{corollary} 

\begin{proof}
Suppose that $\mathbf{t}$ is of finite type. From \cite[Lemma 2.4]{H} and its proof, there exists $\{T_{\lambda}\}_{\lambda \in \Lambda}$, a set of finitely presented $R$-modules in $\mathcal{T}$, such that $\mathcal{F}=
\{T_{\lambda} : {\lambda \in \Lambda}\}^{\perp}$.
The result now follows from Theorem \ref{Teo. tCG}.
\end{proof}


The following corollary, relaxes the condition in Theorem \ref{Teo. tCG} when the torsion pair is of finite type.
\begin{corollary}\label{Cor. relax}
Let $\mathbf{t}=(\mathcal{T,F})$ be a torsion pair in $R$-Mod of finite type. The following assertions are equivalent.
\begin{enumerate}
\item \label{t-is-tCG}  $\mathbf{t}$ is a {\tCG} torsion pair. 
\item \label{Exists-fin-pres} There exists a set $\{T_{\lambda}\}_{\lambda \in \Lambda}$  of finitely presented $R$-modules in $\mathcal{T}$, such that $\mathcal{F}={\bigcap}_{\lambda \in \Lambda} \text{Ker}(\text{Hom}_{R}(T_{\lambda},?))$.
\item \label{Exists-fin-gen} There exists a set $\{T_{\lambda}\}_{\lambda \in \Lambda}$ of finitely generated $R$-modules in $\mathcal{T}$, such that $\mathcal{F}={\bigcap}_{\lambda \in \Lambda} \text{Ker}(\text{Hom}_{R}(T_{\lambda},?))$.
\end{enumerate} 
\end{corollary}

\begin{proof}
Theorem \ref{Teo. tCG} says that \eqref{t-is-tCG} $\iff$ \eqref{Exists-fin-pres},  and since that \eqref{Exists-fin-pres} $\implies$ \eqref{Exists-fin-gen} is trivial, we just need to prove \eqref{Exists-fin-gen} $\implies$ \eqref{Exists-fin-pres}. This last implication follows directly from the proof of \cite[Lemma 2.4]{H}, nevertheless, for clarity, we include some details of the proof. Let $\{T_{\lambda}\}_{\lambda \in \Lambda}$ be a set of finitely generated $R$-modules in $\T$, such that $\mathcal{F}={\bigcap}_{\lambda \in \Lambda} \text{Ker}(\text{Hom}_{R}(T_{\lambda},?))$. For each $\lambda$, we consider the following exact sequence in $R$-Mod:
\[
{0 \longrightarrow  K \longrightarrow R^{(n_{\lambda})} \longrightarrow T_{\lambda} \longrightarrow 0},
\]
where $n_{\lambda}$ is a natural number. Now, we fix a direct system $(K_i)_{i\in I}$ of finitely generated submodules of $K$ such that $K={\bigcup}_{i \in I} K_i .$ Note that for each $i$, we have the following commutative diagram:
\[
\begin{tikzpicture}[node distance=1.5cm]
\node (0ul) {$0$};
\node (Ki) [right of=0ul] {$K_i$};
\node (Rnu) [right of=Ki] {$R^{(n_{\lambda})}$};
\node (Rn/Ki) at ($(Rnu)+(1.85,0)$) {$R^{(n_{\lambda})}/K_i$};
\node (0ur) [right of=Rn/Ki] {$0$};
\node (0dl) [below of=0ul] {$0$};
\node (K) [below of=Ki] {$K$};
\node (Rnd) [below of=Rnu] {$R^{(n_{\lambda})}$};
\node (T) [below of=Rn/Ki] {$T_{\lambda}$};
\node (0dr) [below of=0ur] {$0$.};
\draw[->] (0ul) -- (Ki);
\draw[->] (Ki) -- (Rnu);
\draw[->] (Rnu) -- (Rn/Ki);
\draw[->] (Rn/Ki) -- (0ur);
\draw[->] (0dl) -- (K);
\draw[->] (K) -- (Rnd);
\draw[->] (Rnd) -- (T);
\draw[->] (T) -- (0dr);
\draw[right hook->] (Ki) -- node [right] {\scriptsize $\iota_i$} (K);
\draw[double equal sign distance] (Rnu) -- (Rnd);
\draw[->>] (Rn/Ki) -- (T);
\end{tikzpicture}
\]
It follows that $T_{\lambda}\cong \varinjlim {R^{(n_{\lambda})}}/{K_i}$, where all morphisms of this direct system are projections. Since $\F=\varinjlim \F$, we obtain that $T_{\lambda}=\varinjlim t({R^{(n_{\lambda})}}/{K_i})$. 

Next, for each $i\in I$, we consider the following commutative diagram: 
\[
\begin{tikzpicture}[node distance=1.5cm]
\node (0ul) {$0$};
\node (Kiu) [right of=0ul] {$K_i$};
\node (J) [right of=Ki] {$J_i$};
\node (t) at ($(J)+(2.75,0)$) {$t(R^{(n_{\lambda})}/K_i)$};
\node (0ur) [right of=t] {$0$};
\node (0dl) [below of=0ul] {$0$};
\node (Kid) [right of=0dl] {$K_i$};
\node (Rnd) [below of=J] {$R^{(n_{\lambda})}$};
\node (R/K) [below of=t] {$R^{(n_{\lambda})}/K_i$};
\node (0dr) [right of=R/K] {$0$};
\node (R/J) [below of=Rnd] {$R^{(n_{\lambda})}/J_i$};
\node (1t) [below of=R/K] {$(1:t)(R^{(n_{\lambda})}/J_i)$.};
\draw[->] (0ul) -- (Kiu);
\draw[->] (Kiu) -- (J);
\draw[->] (J) -- (t);
\draw[->] (t) -- (0ur);
\draw[->] (0dl) -- (Kid);
\draw[->] (Kid) -- (Rnd);
\draw[->] (Rnd) -- (R/K);
\draw[->] (R/K) -- (0dr);
\draw[right hook->] (J) --  (Rnd);
\draw[double equal sign distance] (Kiu) -- (Kid);
\draw[->>] (R/K) -- (1t);
\draw[right hook->] (t) -- (R/K);
\draw[->>] (Rnd) -- (R/J);
\draw[->] (R/J) -- node[above]{$\sim$} (1t);
\node at (3.5,-0.45) {$\ulcorner$};
\end{tikzpicture}
\]
Where the top right square of this diagram is the pullback square obtained from ${R^{(n_{\lambda})} \rightarrow {R^{(n_{\lambda})}}/{K_i}}$ and ${t({R^{(n_{\lambda})}}/{K_i}) \hookrightarrow {R^{(n_{\lambda})}}/{K_i}}$.

Given that $\varinjlim {R^{(n_{\lambda})}}/{J_i}\cong \varinjlim (1:t)({R^{(n_{\lambda})}}/{K_i}) =0$, it follows that $(J_i)_{i\in I}$ is a directed union such that ${\bigcup}_{i \in I} J_i=R^{(n_{\lambda})}$. Since $R^{(n_{\lambda})}$ is a finitely generated $R$-module, there is $k\in I$ such that $J_{i}=R^{(n_{\lambda})}$, for all $i\geq k$. Hence ${R^{(n_{\lambda})}}/{K_i}$ is a finitely presented $R$-module which is in $\T$, for all $i\geq k$. If we let $\mathcal{S}_{\lambda}:=\{{R^{(n_{\lambda})}}/{K_i}: i\geq k\}$, then we obtain that $\F={\bigcap}_{S\in \bigcup \mathcal{S}_{\lambda}} \text{Ker}(\text{Hom}_{R}(S,?))$.
\end{proof}

\begin{theorem}\label{Teo tCG noetherian}
Let $R$ be a left Noetherian ring and let $\mathbf{t}=(\mathcal{T,F})$ be a torsion pair in $R$-Mod. The following assertions are equivalent:
\begin{enumerate}
\item \label{t-is-tCG-2} $\mathbf{t}$ is a {\tCG} torsion pair.
\item \label{t-fin-type} $\mathbf{t}$ is of finite type.
\item \label{Ht-Grot} $\mathcal{H}_{\mathbf{t}}$ is a Grothendieck category.
\end{enumerate}
\end{theorem}

\begin{proof}
By \cite[Theorem 1.2]{PS2} and Corollary \ref{cor. tCG implies F}, we just need to prove \eqref{t-fin-type} $\implies$ \eqref{t-is-tCG-2}. From \cite[Lemma 4.6]{PS1}, we can assume that $\mathcal{T}=\text{Gen}(V)$, for some $R$-module $V$. Now, we fix a direct system $(V_{\lambda})_{\lambda \in \Lambda}$ of finitely generated submodules of $V$, such that $\varinjlim V_{\lambda}=V$. By hypothesis, $\F=\varinjlim \F$, and so we get that $\varinjlim t(V_{\lambda})\cong V$. Furthermore, each $t(V_{\lambda})$ is a finitely presented $R$-module, since $R$ is a left Noetherian ring. Note that $\mathcal{F}=\text{Ker}(\text{Hom}_{R}(V,?))=\text{Ker}(\text{Hom}_{R}(\varinjlim{t(V_{\lambda})},?))$. Therefore $\mathcal{F}={\bigcap}_{\lambda \in \Lambda} \text{Ker}(\text{Hom}_{R}(t(V_{\lambda}),?))$, and hence the result follows from Theorem \ref{Teo. tCG}.
\end{proof}

The next result is a surprising result to the authors that can be obtained as an immediate corollary of Theorem \ref{Teo tCG noetherian} and \cite[Theorem 3.10]{AJS}. However, this  result is also given in \cite[Lemma 4.2]{AH}, but from a different point of view, namely the theory of the \emph{silting modules}.

\begin{corollary}
Let $R$ be a commutative Noetherian ring and $\mathbf{t}=(\mathcal{T,F})$ be a torsion pair in $R\text{-Mod}$. Then $\mathbf{t}=(\mathcal{T,F})$ is a hereditary torsion pair if, and only if, $\mathcal{F}$ is closed under direct limits.
\end{corollary}

\begin{proof}
This follows from Theorem \ref{Teo tCG noetherian} and Example \ref{ejemplo-sp}. Nevertheless, for the benefit of the reader, we make one of the implications explicit.

If $\mathbf{t}=(\T,\F)$ is a hereditary torsion pair in $R$-Mod, then we have that $\F={\bigcap}_{\mathfrak{a}\in F_g} \text{Ker}(\text{Hom}_{R}({R}/{\mathfrak{a}},?))$, where $F_g$ is the Gabriel filter associated to the torsion pair $\mathbf{t}$ (see \cite{S}). Since $R$ is a Noetherian ring, then ${R}/{\mathfrak{a}}$ is finitely presented as an $R$-module and so $\F$ is closed under direct limits. 


\end{proof}

Recall that a ring $R$ is called \emph{left coherent ring} if each finitely generated left ideal of $R$ is a finitely presented $R$-module. It is a well known fact that if $R$ is a left coherent ring, then the subcategory $fp(R\text{-Mod})$ of finitely presented $R$-modules, is an abelian category. Again, using the fact that isomorphism classes of finitely presented modules form a set, we get that all torsion pairs in $fp(R\text{-Mod})$ form a set, which we will denote by $\mathbf{t}(fp(R))$.

The following result shows us a way to get a {\tCG} torsion pair, for left coherent rings, from a torsion pair in $\mathbf{t}(fp(R))$. However there are {\tCG} torsion pairs that cannot be obtained in this way. First we recall that a ring $R$ is called \emph{von Neumann regular ring} (VNR), if for every $a$ in $R$, there exists $x$ in $R$ such that $a=axa$. It is well known that over  a VNR ring every finitely presented $R$-module is a projective $R$-module, hence every VNR ring is also coherent.

\begin{theorem}\label{Teo. Coherent}
Let $R$ be a left coherent ring. The assignment $\mathbf{t}=(\mathcal{X,Y}) \leadsto \tilde{\mathbf{t}}:= (\varinjlim{\mathcal{X}},\varinjlim{\mathcal{Y}})$ defines an injective function ${ \phi: \mathbf{t}(fp(R)) \longrightarrow  \mathbf{t}\mathcal{CG}(R).}$ 

Moreover, the following assertions hold:
\begin{enumerate}
\item \label{phi-bij} If $R$ is a Noetherian ring, then $\phi$ is a bijective function.
\item \label{phi-not-bij} Let $R:={\prod}_{\mathbb{N}} {\mathbb{Z}}/{2\mathbb{Z}}$ with addition and multiplication defined componentwise. For this coherent ring, $\phi$ is not bijective.
\end{enumerate}
\end{theorem}

\begin{proof}
Let $\mathbf{t}=(\mathcal{X,Y})$ be a torsion pair in $fp(R\text{-Mod})$. By \cite[Lemma 4.4]{C}, we know that the torsion pair in $R$-Mod generated by $\mathcal{X}$ is $\tilde{\mathbf{t}}=(\mathcal{T,F})=(\varinjlim \mathcal{X}, \varinjlim \mathcal{Y})$. In other words,  $\mathcal{F}=\varinjlim \mathcal{Y}$ consists of the $R$-modules $F$ such that $\text{Hom}_{R}(X,F)=0$, for all $X\in \mathcal{X}$. Thus, $\F={\bigcap}_{X \in \mathcal{X}} \text{Ker}(\text{Hom}_{R}(X,?))$ and therefore
$\tilde{\mathbf{t}}=(\mathcal{T,F})=(\varinjlim \mathcal{X}, \varinjlim \mathcal{Y})$ is a {\tCG} torsion pair.

On the other hand, if $\mathbf{t}_1=(\mathcal{X}_1,\mathcal{Y}_1)$ is a torsion pair in $fp(R\text{-Mod})$ such that $\tilde{\mathbf{t}_1}=(\varinjlim \mathcal{X}_1, \varinjlim \mathcal{Y}_1)=(\varinjlim \mathcal{X}, \varinjlim \mathcal{Y})=\tilde{\mathbf{t}}$. Then, for each $X_1\in \mathcal{X}_1$, we have that $(1:t)(X_1)\in \mathcal{X}_1 \cap \mathcal{Y}\subseteq \varinjlim \mathcal{X}_1\cap \varinjlim\mathcal{Y}=\varinjlim \mathcal{X}\cap \varinjlim \mathcal{Y}=0$, hence $X_1\in \mathcal{X}$ and therefore $\mathcal{X}_1\subseteq \mathcal{X}$. By a similar argument we obtain that $\mathcal{X}\subseteq \mathcal{X}_1$, so that $\mathbf{t}=\mathbf{t}_1$. 
This shows that $\phi$ is an injective function, thus completing the first part of the theorem.

We now suppose that $R$ is a Noetherian ring. Let $\mathbf{t}=(\mathcal{T,F})$ be a {\tCG} torsion pair in $R$-Mod. We will show that there is $\mathbf{t}_{1}\in \mathbf{t}(fp(R))$, such that $\phi(\mathbf{t}_1)=\mathbf{t}$. If $P\in fp(R\text{-Mod})$, then we have the following exact sequence in $R$-Mod:
\[
0 \longrightarrow t(P) \longrightarrow P \longrightarrow (1:t)(P) \longrightarrow 0.
\]

Note that $t(P)$ is a finitely generated $R$-module, which in this case is also a finitely presented $R$-module.
Therefore, the previous exact sequence is also in $fp(R\text{-Mod})$. 
This shows that $\mathbf{t}_1:=(\mathcal{T}\cap fp(R\text{-Mod}), \mathcal{F} \cap fp(R\text{-Mod}))$ is a torsion pair in $fp(R\text{-Mod})$. 
Since $\mathcal{F}$ is closed under direct limits, we obtain that $\varinjlim (\mathcal{T}\cap fp(R\text{-Mod}))\subseteq \mathcal{T}$ and that $\varinjlim( \mathcal{F}\cap fp(R\text{-Mod}))\subseteq \mathcal{F}$, and hence $\phi(\mathbf{t}_1)=\mathbf{t}$.

For the final statement of the theorem, let $R:={\prod}_{{\mathbb{N}}} {\mathbb{Z}}/{2\mathbb{Z}}$ and let $\mathfrak{a}:={\bigoplus}_{\mathbb{N}} {\mathbb{Z}}/{2\mathbb{Z}}$. It is clear that $\mathfrak{a}$ is a two-sided  idempotent  ideal of $R$ and we consider the TTF triple $(\C_\mathfrak{a},\T_\mathfrak{a},\F_\mathfrak{a})$ in $R$-Mod associated to the ideal $\mathfrak{a}$.  Next, for each $n\in \mathbb{N}$, we consider  ${{\bigoplus}_{j=1}^n}  {\mathbb{Z}}/{2\mathbb{Z}} \xrightarrow{\iota_n} {\bigoplus}_{\mathbb{N}} {\mathbb{Z}}/{2\mathbb{Z}}$, the canonical inclusion, and let $I_n:=\text{Im}(\iota_n)$. Note that $\mathfrak{a}$ is the directed union of the $I_n$, and that every $I_n$ is a finitely presented $R$-module such that $\mathfrak{a}I_n=I_n$. 

Since $\mathcal{T_\mathfrak{a}}=\text{Ker}(\text{Hom}_{R}(\mathfrak{a},?))=\text{Ker}(\text{Hom}_{R}({\bigcup}_{n \in \mathbb{N}} I_n,?))$ and that each $I_n\in \mathcal{C_\mathfrak{a}}$, it follows that $\mathcal{T_\mathfrak{a}}={\bigcap}_{n \in \mathbb{N}} \text{Ker}(\text{Hom}_{R}(I_n,?))$. By Theorem \ref{Teo. tCG}, we obtain that this torsion pair is a {\tCG} torsion pair.

%

We claim that $(\C_\mathfrak{a},\T_\mathfrak{a}) \notin \text{Im}(\phi)$. Indeed, suppose that there exists $\mathbf{t}=(\mathcal{X,Y})$ a torsion pair in $fp(R\text{-Mod})$ such that $\phi(\mathbf{t})=(\C_\mathfrak{a},\T_\mathfrak{a})$. Now, we consider the following exact sequence in $fp(R\text{-Mod})$:
\[
{0 \longrightarrow t(R) \longrightarrow R \longrightarrow (1:t)(R) \longrightarrow 0.}
\]
Since  $t(R)\in \mathcal{X}\subseteq \varinjlim \mathcal{X}= \mathcal{C}_\mathfrak{a}$ and that $(1:t)(R)\in \mathcal{Y}\subseteq \varinjlim \mathcal{Y}= \mathcal{T}_\mathfrak{a}$, we obtain that $t(R)=\mathfrak{a}$ and that $(1:t)(R)={R}/{\mathfrak{a}}$, but $\mathfrak{a}$ is not finitely generated. This is a contradiction.
\end{proof}

\begin{remark} \label{no-tCG-fin-pres}
For the pair $(\C_\mathfrak{a},\T_\mathfrak{a})$ in the proof of the previous theorem, there is no finitely presented module $M$, such that $\text{Gen}(M)=\C_\mathfrak{a}$. Indeed, if such finitely presented module $M$ exists, then it is a direct summand of a finitely generated free module, given that  $R$ is a VNR ring. Furthermore, the fact that $\mathfrak{a} M = M$, implies that there is an integer $n$ such that $e_n  M = 0$, where $e_n$ is the ring element with 1  in the $n$-th component and 0 everywhere else. Hence  $Re_n \not \in \text{Gen}(M)$  and $R e_n \in \C_{\mathfrak{a}}$.
\end{remark}

\section{Construction of {\tCG} torsion pairs and torsion pairs of finite type} \label{S:Construction}

In this section we will present an explicit example of a torsion pair of finite type that is not a {\tCG} torsion pair. 
%
%
%
%
The usual example of a finite type torsion pair is the left constituent pair of a TTF-triple. The following result characterizes which  of these left constituent pairs are {\tCG} torsion pairs. If $N$ is a finitely generated submodule of $M$, then we will write $N \subseteq_{fg} M$.

\begin{theorem}
Let $\mathfrak{a}$ be a two-sided idempotent ideal of $R$ and consider the set $\mathcal{S}_{\mathfrak{a}} = \{ {\mathfrak{a}^{(n)}}/({K \cap \mathfrak{a}^{(n)}}) : n\in \mathbb{N}, K \subseteq_{fg} R^{(n)}  \text{ such that } K + \mathfrak{a}^{(n)} = R^{(n)} \}$. 

Then the torsion pair $(\C_\mathfrak{a},\T_\mathfrak{a})$ is a $t$CG torsion pair if and only if the class $\T_\mathfrak{a} = \bigcap_{S \in \mathcal{S}_{\mathfrak{a}}} \Ker(\Hom (S,?))$.
\end{theorem}

\begin{proof}
We will begin by showing that if $S \in \mathcal{S}_{\mathfrak{a}}$, then $S \in \C_\mathfrak{a} \cap fp(R\text{-Mod})$. Indeed, if $K \subseteq_{fg} R^{(n)}$, with $n \in \mathbb{N}$, such that $K + \mathfrak{a}^{(n)} = R^{(n)}$, then 
\[
\frac{R^{(n)}}{K} = \frac{K + \mathfrak{a}^{(n)}}{K} \cong \frac{\mathfrak{a}^{(n)}}{K \cap \mathfrak{a}^{(n)}}.
\]
Since $K$ is finitely generated, then  ${\mathfrak{a}^{(n)}}/{K \cap \mathfrak{a}^{(n)}}$ is finitely presented and it is also in $\text{Gen}(\mathfrak{a}) = \C_\mathfrak{a}$. Hence, if $\T_\mathfrak{a} = \bigcap_{S \in \mathcal{S}_{\mathfrak{a}}} \Ker(\Hom (S,?))$, then by Theorem \ref{Teo. tCG}, the pair $(\C_\mathfrak{a},\T_\mathfrak{a})$ is a $t$CG torsion pair.

For the converse, let's assume that the pair $(\C_\mathfrak{a},\T_\mathfrak{a})$ is a $t$CG torsion pair. Again by Theorem \ref{Teo. tCG} there exists a set $\{ C_{\lambda} \}_{\lambda \in \Lambda}$ in $\C_\mathfrak{a} \cap fp(R\text{-Mod})$, such that $ \T_{\mathfrak{a}} = \bigcap_{\lambda \in \Lambda} \Ker (\Hom (C_{\lambda},?))$. Then for each $\lambda$ we have the following commutative diagram:
\[
\begin{tikzpicture}[node distance=1.5cm]
\node (01) {$0$};
\node (Ka) [right of=01] {$\bar{K}_{\lambda}$};
\node (a) [right of=Ka] {$\mathfrak{a}^{(n_{\lambda})}$};
\node (Ca) [right of=a] {$C_{\lambda}$};
\node (02) [right of=Ca] {$0$};
\node (03) [below of=01] {$0$};
\node (Kb) [right of=03] {${K}_{\lambda}$};
\node (R) [right of=Kb] {$R^{(n_{\lambda})}$};
\node (Cb) [right of=R] {$C_{\lambda}$};
\node (04) [right of=Cb] {$0$,};
\draw[->] (01) to (Ka);
\draw[->] (Ka) to (a);
\draw[->] (a) to node[above]{\tiny $\bar{q}_{\lambda}$} (Ca);
\draw[->] (Ca) to (02);
\draw[->] (03) to (Kb);
\draw[->] (Kb) to (R);
\draw[->] (R) to node[above]{\tiny $q_{\lambda}$} (Cb);
\draw[->] (Cb) to (04);
\draw[right hook->] (Ka) to (Kb);
\draw[right hook->] (a) to node[right]{\tiny $\iota_{n_{\lambda}}$} (R);
\draw[double equal sign distance] (Ca) to (Cb);
\end{tikzpicture}
\]
where $q_{\lambda}$ is the canonical epimorphism and $\iota_{n_{\lambda}}$ is the canonical inclusion, for some $n_{\lambda} \in \mathbb{N}$. Note that $\bar{q}_{\lambda}$ is an epimorphism, since $\mathfrak{a}C_{\lambda} = C_{\lambda}$, and it follows that $K_{\lambda} + \mathfrak{a}^{(n_{\lambda})} = R^{(n_{\lambda})}$. So $C_{\lambda} \in \mathcal{S}_{\mathfrak{a}}$, for each $\lambda \in \Lambda$.  To complete the proof we observe that: 
\[
\T_{\mathfrak{a}} \subseteq \bigcap_{S \in \mathcal{S}_{\mathfrak{a}}}  \Ker( \Hom(S,?)) \subseteq \bigcap_{\lambda \in \Lambda} \Ker (\Hom (C_{\lambda},?)) = \T_{\mathfrak{a}}.
\]

\end{proof}

Given that in the previous result we are dealing with the following module identity $K+\mathfrak{a}^{(n)} = R^{(n)}$, it seems unavoidable to investigate the relation of $\mathfrak{a}$  with superflous ideals. Recall, that an ideal $I$ of a ring $R$ is said to be \emph{superflous} if $I+J = R$, for any other ideal $J$, then $J=R$. Indeed, we have the following result, which also provides an example of a torsion pair of finite type that is not a {\tCG} torsion pair. 

\begin{corollary}
Let $\mathfrak{a}$ be a two-sided idempotent ideal of a ring $R$. If $\mathfrak{a}$ is not finitely generated and contained in the Jacobson radical of $R$, then $\C_{\mathfrak{a}} \cap fp(\text{$R$-Mod}) = 0$. This implies that $(\C_\mathfrak{a},\T_\mathfrak{a})$ is not a $t$CG  torsion pair.
\end{corollary}

\begin{proof}
Given that $\mathfrak{a}$ is contained in the Jacobson radical, then Nakayama's lemma tells us that if $\mathfrak{a}$ is finitely generated, then $\mathfrak{a}=0$. So suppose $\mathfrak{a}$ is not finitely generated, which implies that the torsion pair $(\C_\mathfrak{a},\T_\mathfrak{a})$ is not trivial.

Let $S \in \C_{\mathfrak{a}} \cap fp(\text{$R$-Mod})$, then by the proof of the previous theorem, we have that: 
\[
S\cong \frac{\mathfrak{a}^{(n)}}{K \cap \mathfrak{a}^{(n)}},
\]
for some $n \in \mathbb N$ and some $K \subseteq_{fg} R^{(n)}$, such that $K + \mathfrak{a}^{(n)} = R^{(n)}$. This last equality shows that $\pi_{n}(K) +\mathfrak{a}=R$, where $\pi_n$ denotes the $n$-th projection. Since, by assumption, $\mathfrak{a}$ is a superflous ideal of $R$, we have that $\pi_n(K) = R$ and the following commutative diagram: 

\[
\begin{tikzpicture}[node distance=1.5cm]
\node (0ul) {$0$};
\node (Kiu) [right of=0ul] {$K_{n-1}$};
\node (J) [right of=Ki] {$R^{(n-1)}$};
\node (t)  [right of=J] {$S$};
\node (0ur) [right of=t] {$0$};
\node (0dl) [below of=0ul] {$0$};
\node (Kid) [right of=0dl] {$K$};
\node (Rnd) [below of=J] {$R^{(n)}$};
\node (R/K) [below of=t] {$S$};
\node (0dr) [right of=R/K] {$0$};
\node (R/J) [below of=Rnd] {$R$};
\node (1t) [below of=Kid] {$\pi_n(K)$};
\draw[->] (0ul) -- (Kiu);
\draw[->] (Kiu) -- (J);
\draw[->] (J) -- (t);
\draw[->] (t) -- (0ur);
\draw[->] (0dl) -- (Kid);
\draw[->] (Kid) -- (Rnd);
\draw[->] (Rnd) -- (R/K);
\draw[->] (R/K) -- (0dr);
\draw[right hook->] (J) --  (Rnd);
\draw[right hook->] (Kiu) -- (Kid);
\draw[->>] (Kid) -- node[right]{$\pi_n$}(1t);
\draw[double equal sign distance] (t) -- (R/K);
\draw[->>] (Rnd) -- node[right]{$\pi_n$} (R/J);
\draw[double equal sign distance] (R/J) --  (1t);
%
\end{tikzpicture}
\]
Since $\mathfrak{a}S=S$, we deduce that $K_{n-1}+\mathfrak{a}^{(n-1)}=R^{(n-1)}$. Using this argument in a recursive manner, we obtain that $K_1+\mathfrak{a}=R$ and that $R/K_1\cong S$. Once again, the ideal $\mathfrak{a}$ is superflous,  hence  $K_1=R$, which in turn gives us that $S=0$. The final statement of the corollary follows from Theorem \ref{Teo. tCG}.
\end{proof}

\begin{example}
Consider, $R = C([0,1])$, the commutative ring of continuous real valued functions over the unit interval. For  $x \in [0,1]$  we consider a maximal ideal $M_x$ given by the functions in $R$ such that its value at $x$ is 0. Now localize $R$ at $M_x$, and  let $\mathfrak{a} = {M_x}_{M_x}$, which is an idempotent, two-sided, not finitely generated ideal, and observe that the Jacobson radical of $R_{M_x}$ is precisely ${M_x}_{M_x}$. Applying the previous corollary, we get that the torsion pair $(\C_\mathfrak{a},\T_\mathfrak{a})$ is not a $t$CG torsion pair in $R_{M_x}$-Mod.
\end{example}

\begin{lemma}\label{lem. necesario}
Let $R$ be a VNR ring. If $\mathfrak{a}$ is a two-sided ideal of $R$, then for each $a\in \mathfrak{a}$, we get $\mathfrak{a}(Ra)=Ra$.
\end{lemma}

\begin{proof}
Note that $\mathfrak{a}(Ra)=(\mathfrak{a}R)a=\mathfrak{a}a$. On the other hand, there exists $x$ in $R$, such that $axa=a$. Thus, for every $r\in R$, we get $ra=(rax)a$ and since $\mathfrak{a}$ is a two-sided ideal of $R$, we obtain that $ra\in \mathfrak{a}a$, so that $Ra\subset \mathfrak{a}a$. This shows that $Ra\subseteq \mathfrak{a}a=\mathfrak{a}(Ra)\subseteq Ra$. 
\end{proof}

The following result gives us an explicit description of {\tCG} torsion pairs over a VNR ring.

\begin{theorem}\label{Teo. Von}
Let $R$ be a VNR ring and let $\mathbf{t}=(\T,\F)$ be a torsion pair in $R$-Mod. The following assertions are equivalent:
\begin{enumerate}
\item \label{t-is-tCG-3} $\mathbf{t}$ is a {\tCG} torsion pair;
\item \label{exists-fin-gen} There exists $\{T_{\lambda}\}_{\lambda \in \Lambda}$, a set of finitely generated projective $R$-modules such that $\mathcal{T}=\text{Gen}(\coprod T_{\lambda})$;
\item \label{idempotent} There exists a unique idempotent two-sided ideal $\mathfrak{a}$ of $R$ such that $\T = \C_\mathfrak{a}=\text{Gen}(\mathfrak{a})=\{T  \in  R\text{-Mod}: \mathfrak{a}T=T\}$;
\item \label{left-const} $\mathbf{t}$ is the left constituent pair of a TTF triple in $R$-Mod. 
\end{enumerate}
\end{theorem}

\begin{proof}
\eqref{t-is-tCG-3} $\implies$ \eqref{exists-fin-gen}. By Theorem \ref{Teo. tCG}, there exists $\{T_{\lambda}\}_{\lambda \in \Lambda}$, a set of finitely presented $R$-modules in $\T$, such that $\F={\bigcap}_{\lambda \in \Lambda} \text{Ker}(\text{Hom}_{R}(T_{\lambda},?))$. Since $R$ is a VNR ring, we get that each $T_{\lambda}$ is a finitely generated projective $R$-module. Therefore, $\coprod T_{\lambda}$ is also a projective $R$-module and hence $\text{Gen}(\coprod T_{\lambda})$ is a torsion class in $R$-Mod. In this case, the equality $\F={\bigcap}_{\lambda \in \Lambda} \text{Ker}(\text{Hom}_{R}(T_{\lambda},?))$, implies $\T=\text{Gen}(\coprod T_{\lambda})$. 

\noindent \eqref{exists-fin-gen} $\implies$ \eqref{idempotent}. It follows from \cite[Proposition VI.9.4]{S} and \cite[Corollary VI.9.5]{S}. 

\noindent \eqref{idempotent} $\implies$ \eqref{left-const}. It follows from \cite[Chapter VI, Section 8]{S}. 

\noindent \eqref{left-const} $\implies$ \eqref{t-is-tCG-3}. We suppose that there exists a TTF triple in $R$-Mod of the form $(\mathcal{T},\mathcal{F},\mathcal{F}^{\perp})$. Hence, $\mathbf{t}=\mathcal{(T,F)}$ is a torsion pair of finite type, and there exists a unique idempotent two-sided ideal $\mathfrak{a}$ of $R$, such that $\T = \C_\mathfrak{a}=\text{Gen}(\mathfrak{a})=\{T \in R\text{-Mod}: \mathfrak{a}T=T\}$. Now, we fix a direct system $(I_i)_{i \in I}$ of finitely generated submodules of $\mathfrak{a}$ such that $\mathfrak{a}={\bigcup}_{i \in I}I_i$. By Lemma \ref{lem. necesario}, we have $\mathfrak{a}I_i=I_i$, for all $i\in I$, hence each $I_i$ is in $\mathcal{T} = \C_{\mathfrak{a}}$. The result follows from Corollary \ref{Cor. relax}, since $\F=\text{Ker}(\text{Hom}_{R}(\mathfrak{a},?))=\text{Ker}(\text{Hom}_{R}({\bigcup}_{i \in I} I_i,?))={\bigcap}_{i \in I} \text{Ker}(\text{Hom}_{R}(I_i,?))$.
\end{proof}

Given a torsion pair in $R$-Mod  of the form $(\text{Gen}(V) , \text{Ker}({\Hom_R(V,?)} ))$ and a direct system of finitely presented modules $\{ V_{\lambda} \}_{\lambda \in \Lambda}$ with $V = \varinjlim V_{\lambda}$, we get from  \cite[Example VI.2.3]{S} that the pair $(^{\perp}(\bigcap \text{Ker}(\Hom_R(V_{\lambda},?))),\bigcap \text{Ker}(\Hom_R(V_{\lambda},?)))$ is a torsion pair. From Theorem \ref{Teo. tCG} we get that this last torsion pair is a {\tCG} torsion pair; furthermore, we get that $\text{Gen}(V) \subseteq {}^{\perp}(\bigcap \text{Ker}(\Hom_R(V_{\lambda},?)))$. This fact motivates the following question.
\begin{question}
Given $R$ a ring and  $\mathbf{t}=(\mathcal{T,F})$ a torsion pair in $R$-Mod, such that $\mathcal{T}=\text{Gen}(V)$ for some $R$-module $V$. Is there  a minimal torsion class $\mathcal{T}'$ of a {\tCG} torsion pair, such that $\mathcal{T} \subseteq \mathcal{T}'$?
\end{question}

\section*{Acknowledgements}

The authors thank Manuel Saor\'in for several helpful discussions, and the comments of the referee which certainly have  improved  the quality of this paper.

\bibliographystyle{amsalpha}
\bibliography{tCG-biblio}

\providecommand{\bysame}{\leavevmode\hbox to3em{\hrulefill}\thinspace}
\providecommand{\MR}{\relax\ifhmode\unskip\space\fi MR }
\providecommand{\MRhref}[2]{%
  \href{http://www.ams.org/mathscinet-getitem?mr=#1}{#2}
}
\providecommand{\href}[2]{#2}
\begin{thebibliography}{ATJLSS03}

\bibitem[AHH17]{AH}
Lidia Angeleri~H\"ugel and Michal Hrbek, \emph{Silting modules over commutative
  rings}, Int. Math. Res. Not. IMRN (2017), no.~13, 4131--4151. \MR{3671512}

\bibitem[ATJLS10]{AJS}
Leovigildo Alonso~Tarr{\'{\i}}o, Ana Jerem{\'{\i}}as~L{\'o}pez, and Manuel
  Saor{\'{\i}}n, \emph{Compactly generated {$t$}-structures on the derived
  category of a {N}oetherian ring}, J. Algebra \textbf{324} (2010), no.~3,
  313--346. \MR{2651339}

\bibitem[ATJLSS03]{AJSo}
Leovigildo Alonso~Tarr{\'{\i}}o, Ana Jerem{\'{\i}}as~L{\'o}pez, and
  Mar{\'{\i}}a~Jos{\'e} Souto~Salorio, \emph{Construction of {$t$}-structures
  and equivalences of derived categories}, Trans. Amer. Math. Soc. \textbf{355}
  (2003), no.~6, 2523--2543 (electronic). \MR{1974001}

\bibitem[BBD82]{BBD}
A.~A. Be{\u\i}linson, J.~Bernstein, and P.~Deligne, \emph{Faisceaux pervers},
  Analysis and topology on singular spaces, {I} ({L}uminy, 1981), Ast\'erisque,
  vol. 100, Soc. Math. France, Paris, 1982, pp.~5--171. \MR{751966}

\bibitem[CB94]{C}
William Crawley-Boevey, \emph{Locally finitely presented additive categories},
  Comm. Algebra \textbf{22} (1994), no.~5, 1641--1674. \MR{1264733}

\bibitem[CGM07]{CGM}
Riccardo Colpi, Enrico Gregorio, and Francesca Mantese, \emph{On the heart of a
  faithful torsion theory}, J. Algebra \textbf{307} (2007), no.~2, 841--863.
  \MR{2275375}

\bibitem[CMT11]{CMT}
Riccardo Colpi, Francesca Mantese, and Alberto Tonolo, \emph{When the heart of
  a faithful torsion pair is a module category}, J. Pure Appl. Algebra
  \textbf{215} (2011), no.~12, 2923--2936. \MR{2811575}

\bibitem[Dic66]{D}
Spencer~E. Dickson, \emph{A torsion theory for {A}belian categories}, Trans.
  Amer. Math. Soc. \textbf{121} (1966), 223--235. \MR{0191935}

\bibitem[HKM02]{HKM}
Mitsuo Hoshino, Yoshiaki Kato, and Jun-Ichi Miyachi, \emph{On {$t$}-structures
  and torsion theories induced by compact objects}, J. Pure Appl. Algebra
  \textbf{167} (2002), no.~1, 15--35. \MR{1868115}

\bibitem[Hrb16]{H}
Michal Hrbek, \emph{One-tilting classes and modules over commutative rings}, J.
  Algebra \textbf{462} (2016), 1--22. \MR{3519496}

\bibitem[HRS96]{HRS}
Dieter Happel, Idun Reiten, and Sverre~O. Smal{\o}, \emph{Tilting in abelian
  categories and quasitilted algebras}, Mem. Amer. Math. Soc. \textbf{120}
  (1996), no.~575, viii+ 88. \MR{1327209}

\bibitem[MT12]{MT}
Francesca Mantese and Alberto Tonolo, \emph{On the heart associated with a
  torsion pair}, Topology Appl. \textbf{159} (2012), no.~9, 2483--2489.
  \MR{2921836}

\bibitem[PS15]{PS1}
Carlos~E. Parra and Manuel Saor{\'{\i}}n, \emph{Direct limits in the heart of a
  t-structure: the case of a torsion pair}, J. Pure Appl. Algebra \textbf{219}
  (2015), no.~9, 4117--4143. \MR{3336001}

\bibitem[PS16a]{PS2}
\bysame, \emph{Addendum to ``{D}irect limits in the heart of a t-structure: the
  case of a torsion pair'' [{J}. {P}ure {A}ppl. {A}lgebra 219 (9) (2015)
  4117--4143] [ {MR}3336001]}, J. Pure Appl. Algebra \textbf{220} (2016),
  no.~6, 2467--2469. \MR{3448805}

\bibitem[PS16b]{PS3}
\bysame, \emph{On hearts which are module categories}, J. Math. Soc. Japan
  \textbf{68} (2016), no.~4, 1421--1460.

\bibitem[PS17]{PS4}
\bysame, \emph{Hearts of t-structures in the derived category of a commutative
  {N}oetherian ring}, Trans. Amer. Math. Soc. \textbf{369} (2017), no.~11,
  7789--7827. \MR{3695845}

\bibitem[Ric89]{R}
Jeremy Rickard, \emph{Morita theory for derived categories}, J. London Math.
  Soc. (2) \textbf{39} (1989), no.~3, 436--456. \MR{1002456}

\bibitem[Ste75]{S}
Bo~Stenstr{\"o}m, \emph{Rings of quotients}, Springer-Verlag, New
  York-Heidelberg, 1975, Die Grundlehren der Mathematischen Wissenschaften,
  Band 217, An introduction to methods of ring theory. \MR{0389953}

\end{thebibliography}
\end{document}